\documentclass[12pt,reqno]{amsart}
\usepackage{amsmath,amssymb,amsfonts,amscd,latexsym,amsthm,mathrsfs}
\usepackage[usenames]{color}
\usepackage[unicode]{hyperref}
\makeatletter
\renewcommand\thesubsection{\@arabic\c@subsection}
\makeatother

\let\ol\overline
\renewcommand{\d}{{\mathrm d}}

\newcommand{\Li}{\operatorname{Li}}
\renewcommand{\Re}{\operatorname{Re}}
\renewcommand{\Im}{\operatorname{Im}}

\newtheorem{theorem}{Theorem}
\newtheorem{lemma}{Lemma}

\theoremstyle{remark}

\begin{document}

\title{A study of elliptic gamma function and allies}

\date{31 December 2017. \emph{Revised}: 25 April 2018}

\author{Vicen\c tiu Pa\c sol}
\address{Simion Stoilow Institute of Mathematics of the Romanian Academy, P.O.\ Box 1-764, 014700 Bucharest, Romania}
\email{vicentiu.pasol@imar.ro}

\author{Wadim Zudilin}
\address{Department of Mathematics, IMAPP, Radboud University, PO Box 9010, 6500~GL Nijmegen, Netherlands}
\email{w.zudilin@math.ru.nl}

\address{School of Mathematical and Physical Sciences, The University of Newcastle, Callag\-han, NSW 2308, Australia}
\email{wadim.zudilin@newcastle.edu.au}

\address{Laboratory of Mirror Symmetry and Automorphic Forms, National Research University Higher School of Economics, 6 Usacheva str., 119048 Moscow, Russia}
\email{wzudilin@gmail.com}

\dedicatory{To Don Zagier, in admiration of his insights on modular, elliptic and polylogarithmic functions}

\thanks{The first author was partially supported by a grant of Romanian Ministry of Research and Innovation,
CNCS\,--\,UEFISCDI, project number PN-III-P4-ID-PCE-2016-0157, within PNCDI III7.
The second author is partially supported by Laboratory of Mirror Symmetry NRU HSE, RF government grant, ag.\ no.\ 14.641.31.0001.}


\keywords{Theta function; elliptic gamma function; elliptic dilogarithm; elliptic polylogarithm}

\begin{abstract}
We study analytic and arithmetic properties of the elliptic gamma function
$$
\prod_{m,n=0}^\infty\frac{1-x^{-1}q^{m+1}p^{n+1}}{1-xq^mp^n},
\qquad |q|,|p|<1,
$$
in the regime $p=q$; in particular, its connection with the elliptic dilogarithm and a formula of S.~Bloch.
We further extend the results to more general products by linking them to non-holomorphic Eistenstein series
and, via some formulae of D.~Zagier, to elliptic polylogarithms.
\end{abstract}

\maketitle

\section{Introduction}
\label{sec1}

For complex $z$ and $\tau$ with $\Im\tau>0$, set $x=e^{2\pi iz}$ and $q=e^{2\pi i\tau}$.
Transformation properties of the so-called \emph{short} theta-function
$$
\theta_0(z;\tau):=\prod_{m=0}^\infty(1-x^{-1}q^{m+1})(1-xq^m)
$$
under the action of the modular group are well understood. In view of its transparent invariance under translation
$\tau\mapsto\tau+1$, the main source of the modular action originates from the $\tau$-involution
\begin{equation}
z\mapsto\hat z=\frac z\tau, \quad \tau\mapsto\hat\tau=-\frac1\tau.
\label{invo}
\end{equation}
The related classical transformation of $\theta_0(z;\tau)$ can be recorded as
\begin{equation}
q^{1/12}x^{-1/2}\theta_0(z;\tau)
=ie^{-\pi iz\hat z}\hat q^{1/12}\hat x^{-1/2}\theta_0(\hat z;\hat\tau)
\label{trans1}
\end{equation}
(see, for example, \cite[Section 2]{FV00}), where we define $\hat x=e^{2\pi i\hat z}$ and $\hat q=e^{2\pi i\hat\tau}$.

Less is known about modular properties of the related product
\begin{equation*}
\theta_1(z;\tau):=\prod_{m=0}^\infty\frac{(1-x^{-1}q^{m+1})^{m+1}}{(1-xq^m)^m},
\end{equation*}
which naturally comes as the $\sigma=\tau$ specialisation of the elliptic gamma function
$$
\Gamma(z;\tau,\sigma):=\prod_{m,n=0}^\infty\frac{1-x^{-1}q^{m+1}p^{n+1}}{1-xq^mp^n},
\quad\text{where}\; p=e^{2\pi i\sigma},
$$
introduced by S.~Ruijsenaars in \cite{Ru97} (see also~\cite{FV00,FV05}).
Namely, we have
$$
\theta_1(z;\tau)
=\theta_0(z;\tau)\Gamma(z;\tau,\tau)=\Gamma(z+\tau;\tau,\tau).
$$
A known functional equation of the elliptic gamma function \cite[Theorem 4.1]{FV00}
represents an $\operatorname{SL}_3(\mathbb Z)$ symmetry of $\Gamma(z;\tau,\sigma)$.
The problem of determining its behaviour in the regime $\sigma=\tau$ under $\operatorname{SL}_2(\mathbb Z)$ transformations is specifically addressed in \cite{DI07}, where the (logarithm of the) infinite product
is related to the elliptic dilogarithm via a formula of S.~Bloch \cite{Bl77}.

Our principal aim in this note is re-casting analytic and arithmetic (modular) properties of the
function $\theta_1(z;\tau)$ and its relatives, in particular, linking them to non-holomorphic Eisenstein series
and the elliptic dilogarithm. This programme is carried out in Sections~\ref{sec2}--\ref{sec4}; it gives
a new proof of Bloch's formula and related results from~\cite{DI07}. In Section~\ref{sec5} we go further
to discuss similar features of products that generalise ones for $\theta_0$ and $\theta_1$;
their relationship with non-holomorphic Eisenstein series and formulae from~\cite{Za90} allow us
to link them to elliptic polylogarithms.

For future record, notice that iterating the transformation $(z,\tau)\mapsto(\hat z,\hat\tau)$ twice maps $(z,\tau)$ to $(-z,\tau)$ and that
\begin{equation}
\theta_1(-z;\tau)=\frac1{\theta_1(z;\tau)}
\quad\text{and}\quad
\theta_0(-z;\tau)=-x^{-1}\theta_0(z;\tau).
\label{-z}
\end{equation}

\section{Period functions}
\label{sec2}

A natural way of measuring failure of weight $k$ modular behaviour under the transformation $(z,\tau)\mapsto(\hat z,\hat\tau)$ for a function $f(z,\tau)$
is through the \emph{period} function
$$
g(z,\tau)=g_k(z,\tau):=f(\hat z,\hat\tau)-\tau^k f(z,\tau).
$$

\begin{lemma}
\label{lem1}
We have
$$
\tau^k g(\hat z,\hat\tau)+(-1)^k g(z,\tau)
=\tau^k\bigl(f(-z,\tau)-(-1)^kf(z,\tau)\bigr).
$$
\end{lemma}

Observe that the expression in the parentheses on the right-hand side measures the failure of $k$-parity of $f(z,\tau)$.

\begin{proof}
We only use $(\Hat{\hat z},\Hat{\hat\tau})=(-z,\tau)$ and $\tau\hat\tau=-1$:
\begin{align*}
\tau^k g(\hat z,\hat\tau)- g(z,\tau)
&=\tau^k\bigl(f(-z,\tau)-{\hat\tau}^kf(\hat z,\hat\tau)\bigr)+(-1)^k\bigl(f(\hat z,\hat\tau)-\tau^k f(z,\tau)\bigr)
\\
&=\tau^k\bigl(f(-z,\tau)-(-1)^kf(z,\tau)\bigr).
\qedhere
\end{align*}
\end{proof}

The lemma and the parity relation for $\ln\theta_1(z;\tau)$ in \eqref{-z} imply the following.

\begin{lemma}
\label{lem2}
The function
\begin{equation}
T(z;\tau)=\tau\ln\theta_1(z;\tau)-\ln\theta_1(\hat z;\hat\tau)
\label{T}
\end{equation}
satisfies the functional equation
$$
T(\hat z;\hat\tau)=\tau^{-1}T(z;\tau).
$$
\end{lemma}

Furthermore, we can relate the function $T(z;\tau)$ to the the dilogarithm function
$$
\Li_2(x)=-\int_0^x\ln(1-t)\,\frac{\d t}t.
$$

\begin{lemma}
\label{lem3}
The function \eqref{T} admits the following representation:
\begin{align*}
T(z;\tau)
&=\frac{\pi i(\tau-2z)(1+2\tau z-2z^2)}{12\tau}
+z\ln\theta_0(z;\tau)
\\ &\qquad
-\frac1{2\pi i}\sum_{m=0}^\infty\bigl(\Li_2(x^{-1}q^{m+1})-\Li_2(xq^m)\bigr).
\end{align*}
\end{lemma}

\begin{proof}
As shown in the proof of Theorem~5.2 in \cite{FV00},
\begin{align*}
\ln\theta_1(z;\tau)
&=\ln\theta_0(z;\tau)+\ln\Gamma(z;\tau,\tau)
\\
&=-\pi i\lambda(z;\tau)+\ln\frac{\theta_0(z;\tau)}{\theta_0(\hat z;\hat\tau)}
\\ &\qquad
+(\hat\tau-\hat z)\sum_{k=1}^\infty\frac{(\hat x^{-1}\hat q)^k}{k(1-\hat q^k)}
-\hat z\sum_{k=1}^\infty\frac{\hat x^k}{k(1-\hat q^k)}
\\ &\qquad
+\frac1{2\pi i}\sum_{k=1}^\infty\frac{\hat x^k-(\hat x^{-1}\hat q)^k}{k^2(1-\hat q^k)}
-\hat\tau\sum_{k=1}^\infty\frac{\hat q^k(\hat x^k-(\hat x^{-1}\hat q)^k)}{k(1-\hat q^k)^2},
\end{align*}
where
$$
\lambda(z;\tau)=\frac{z^3}{3\tau^2}-\frac{2\tau-1}{2\tau^2}\,z^2+\frac{(\tau-1)(5\tau-1)}{6\tau^2}\,z-\frac{(\tau-2)(2\tau-1)}{12\tau}
$$
and the assumptions $|\hat x|,|\hat x^{-1}\hat q|<1$ are made to ensure convergence.
(The latter can be dropped in the final result by appealing to the analytic continuation in~$z$.)
Recalling the transformation \eqref{trans1}, using
$$
\frac1{1-\hat q^k}=\sum_{m=0}^\infty\hat q^{mk}
\quad\text{and}\quad
\frac{\hat q^k}{(1-\hat q^k)^2}=\sum_{m=0}^\infty m\hat q^{mk},
$$
interchanging summation and summing over $k$, we obtain
\begin{align*}
\ln\theta_1(z;\tau)
&=-\pi i\biggl(\lambda(z;\tau)-\frac12+\frac{z^2}\tau+\frac\tau6-z+\frac1{6\tau}+\frac z\tau\biggr)
\\ &\qquad
+\hat z\sum_{m=0}^\infty\bigl(\ln(1-\hat x^{-1}\hat q^{m+1})+\ln(1-\hat x\hat q^m)\bigr)
\\ &\qquad
-\hat\tau\sum_{m=0}^\infty\bigl((m+1)\ln(1-\hat x^{-1}\hat q^{m+1})-m\ln(1-\hat x\hat q^m)\bigr)
\\ &\qquad
-\frac1{2\pi i}\sum_{m=0}^\infty\bigl(\Li_2(\hat x^{-1}\hat q^{m+1})-\Li_2(\hat x\hat q^m)\bigr)
\displaybreak[2]\\
&=\frac{\pi i}{12}\biggl((1+2z)-\frac{2z(1+z)(1+2z)}{\tau^2}\biggr)
+\hat z\ln\theta_0(\hat z;\hat\tau)
-\hat\tau\ln\theta_1(\hat z;\hat\tau)
\\ &\qquad
-\frac1{2\pi i}\sum_{m=0}^\infty\bigl(\Li_2(\hat x^{-1}\hat q^{m+1})-\Li_2(\hat x\hat q^m)\bigr).
\end{align*}
(This formula can be alternatively derived from logarithmically differentiating identity \eqref{trans1} with respect to~$\tau$
and further integrating the result with respect to~$z$.)
Substituting $(z/\tau,-1/\tau)$ for $(z,\tau)$ translates the result into
\begin{align*}
\tau\ln\theta_1(z;\tau)-\ln\theta_1(\hat z;\hat\tau)
&=\frac{\pi i(\tau-2z)(1+2\tau z-2z^2)}{12\tau}
+z\ln\theta_0(z;\tau)
\\ &\qquad
-\frac1{2\pi i}\sum_{m=0}^\infty\bigl(\Li_2(x^{-1}q^{m+1})-\Li_2(xq^m)\bigr),
\end{align*}
the desired relation.
\end{proof}

\section{Non-holomorphic modularity}
\label{sec3}

Denote
$$
A=A(z,\tau):=\frac{z-\ol z}{\tau-\ol\tau}\in\mathbb R,
$$
so that
$$
\hat A=A(\hat z,\hat\tau):=\frac{z\ol\tau-\ol z\tau}{\tau-\ol\tau}\in\mathbb R
$$
and $z=A\tau-\hat A$.
Define
\begin{equation}
Q(z;\tau):=q^{B_3(A)/3}\prod_{m=0}^\infty\frac{(1-xq^m)^{m+A}}{(1-x^{-1}q^{m+1})^{m+1-A}}
=\frac{q^{B_3(A)/3}\theta_0(z;\tau)^A}{\theta_1(z;\tau)},
\label{Qprod}
\end{equation}
where $B_3(t):=t^3-\frac32t^2+\frac12t$ is the third Bernoulli polynomial, $B_3(1-t)=-B_3(t)$, and
$$
F_+(z;\tau):=\ln Q(\hat z;\hat\tau)-\tau\ln Q(z;\tau),
\quad
F_-(z;\tau):=\ln\ol{Q(\hat z;\hat\tau)}-\tau\ln\ol{Q(z;\tau)}.
$$

It follows then from Lemma~\ref{lem1} and the parity relations \eqref{-z} that
\begin{align*}
\tau F_+(\hat z;\hat\tau)-F_+(z;\tau)
&=\tau\bigl(\ln Q(-z;\tau)+\ln Q(z;\tau)\bigr)
\\
&=\frac{2\pi i}3(B_3(-A)+B_3(A))\tau^2+2\pi iA z\tau-\pi i A\tau
\\
&=-\pi iA\bigl(2(A\tau-z)+1\bigr)\tau
=-\pi i A(2\hat A+1)\tau
\\ \intertext{and}
\tau F_-(\hat z;\hat\tau)-F_-(z;\tau)
&=\tau\bigl(\ln\ol{Q(-z;\tau)}+\ln\ol{Q(z;\tau)}\bigr)
\\
&=-\frac{2\pi i}3(B_3(-A)+B_3(A))\tau\ol\tau-2\pi iA\ol z\tau+\pi i A\tau
\\
&=\pi iA\bigl(2(A\ol\tau-\ol z)+1\bigr)\tau
=\pi iA(2\hat A+1)\tau.
\end{align*}
We summarise our finding in the following claim.

\begin{lemma}
\label{lem4}
We have
\begin{align*}
\tau F_+(\hat z;\hat\tau)-F_+(z;\tau)
&=-\pi i A(2\hat A+1)\tau,
\\
\tau F_-(\hat z;\hat\tau)-F_-(z;\tau)
&=\phantom +\pi iA(2\hat A+1)\tau.
\end{align*}
\end{lemma}

Lemma~\ref{lem3} leads to the following expansions of the functions $F_+$ and $F_-$.

\begin{theorem}
\label{th1}
We have
\begin{align*}
F_+(z;\tau)
&=S(z,\tau)-\frac{1}{2\pi i}\,L(z,\tau),
\\
F_-(z;\tau)
&=-\frac{2\pi i\ol\tau(\tau-\ol\tau)}{3}B_3(A)+\ol{S(z,\tau)}+\frac1\pi\,\ol{U(z,\tau)}+\frac{1}{2\pi i}\,\ol{L(z,\tau)},
\end{align*}
where
\begin{align*}
L(z,\tau)
&:=\sum_{m=0}^\infty\bigl(\Li_2(x^{-1}q^{m+1})-\Li_2(x q^m)\bigr),
\\
U(z,\tau)
&:=\sum_{m=0}^\infty\bigl(\ln|x^{-1}q^{m+1}|\,\Li_1(x^{-1}q^{m+1})-\ln|xq^{m}|\,\Li_1(x q^{m})\bigr),
\\
S(z,\tau)
&:=\frac{-\pi i}{12}(2A-1)(6z^2-12 A\tau z+6z+8 A^2\tau^2-2 A\tau^2-6A\tau+1).
\end{align*}
\end{theorem}

\begin{proof}
For $F_+$ substitute the expression of $T(z;\tau)$ from Lemma~\ref{lem3} into the computation
\begin{align*}
F_+(z;\tau)
&=\ln Q(\hat z;\hat\tau)-\tau\ln Q(z;\tau)
\\
&=\frac{2\pi i}3\bigl(B_3(\hat A)\hat\tau-B_3(A)\tau^2)+\hat A\ln\theta_0(\hat z;\hat\tau)-(\hat A+z)\ln\theta_0(z;\tau)
\\ &\qquad
+\tau\ln\theta_1(z;\tau)-\ln\theta_1(\hat z;\hat\tau).
\end{align*}
This leads to the formula
$$
F_+(z;\tau)
=S(z,\tau)-\frac{1}{2\pi i}\,L(z,\tau)
$$
with
\begin{align*}
S(z,\tau)
&=\frac{2\pi i}{3}\bigl(B_3(\hat A)\hat\tau-B_3(A)\tau^2\bigr)
+\hat A\pi i\biggl(\frac{\tau}{6} -\frac{\hat\tau}{6}+z\hat z-\frac{1}{2}-z+ \hat z\biggr)
\\ &\qquad
+\frac{\pi i}{12\tau}(\tau-2z)(1+2\tau z-2z^2),
\end{align*}
and the latter simplifies to the expression given in the statement of Theorem~\ref{th1} by elementary manipulation.

For $F_-$ we proceed as follows. We have
$$
\ln Q(z;\tau)=\frac{2\pi i\tau B_3(A)}{3}-\sum_{m=0}^\infty\bigl((m+1-A)\Li_1(x^{-1}q^{m+1})-(m+A)\Li_1(x q^{m})\bigr).
$$
Multiply this expression by $\tau-\ol\tau=2i\Im\tau$ and use $A(\tau-\ol\tau)=2i\Im z$ to get
$$
(\tau-\ol\tau)\ln Q(z;\tau)=\frac{2\pi i\tau(\tau-\ol\tau) B_3(A)}{3}-\frac1\pi\, U(z,\tau).
$$
Now, notice
$$
\ol{(\tau-\ol\tau)\ln Q(z;\tau)}=F_-(z;\tau)-\ol{F_+(z;\tau)}
$$
to deduce the expression for $F_-$ as in the theorem.
\end{proof}

A consequence of this expansion is the invariance of
$$
F(z;\tau):=\frac{F_+(z;\tau)+F_-(z;\tau)}2
=\ln|Q(\hat z;\hat\tau)|-\tau\ln|Q(z;\tau)|
$$
under translation $\tau\mapsto\tau+1$.

\begin{lemma}
\label{lem5}
We have
$$
F_+(z;\tau+1)-F_+(z;\tau)=-\bigl(F_-(z;\tau+1)-F_-(z;\tau)\bigr).
$$
\end{lemma}

\begin{proof}
The functions $L(z,\tau)$ and $U(z,\tau)$ (hence their complex conjugates) are clearly invariant under translation $\tau\mapsto\tau+1$.
The result follows from noticing that
\begin{align*}
2 \Re S(z,\tau)+ \frac{2\pi i\ol\tau(\tau-\ol\tau)}{3}\,B_3(A)
&=\frac{-\pi i (\tau-\ol\tau)^2 A(1-A)(1-2A)}{6}
\\
&=\frac{-\pi i (\tau-\ol\tau)^2}3\,B_3(A)
\end{align*}
is also invariant under the transformation.
\end{proof}

We summarise the results in this section as follows.

\begin{theorem}
\label{th2}
The weight $1$ period function
\begin{align*}
F(z;\tau)
&=\ln|Q(\hat z;\hat\tau)|-\tau\ln|Q(z;\tau)|
\\
&=\frac1{2\pi}\sum_{m=0}^\infty\bigl(\ln|x^{-1}q^{m+1}|\,\ol{\Li_1(x^{-1}q^{m+1})}-\ln|xq^{m}|\,\ol{\Li_1(x q^{m})}\,\bigr)
\\ &\qquad
-\frac{\pi i(\tau-\ol\tau)^2}6\,B_3(A)
-\frac1{2\pi i}\,\Im\sum_{m=0}^\infty\bigl(\Li_2(x^{-1}q^{m+1})-\Li_2(x q^m)\bigr)
\end{align*}
of $\ln|Q(z;\tau)|$ satisfies
$$
\tau F(\hat z;\hat\tau)=F(z;\tau)
\quad\text{and}\quad
F(z;\tau)=F(z;\tau+1).
$$
In other words, it behaves like a Jacobi form of weight $1$ on $\operatorname{SL}_2(\mathbb Z)$.
\end{theorem}

\section{Elliptic dilogarithm}
\label{sec4}

Theorem~\ref{th2} provides a natural link between the period function $F(z;\tau)$ and the elliptic dilogarithm \cite{Za90}
$$
D(q;x):=\sum_{m\in\mathbb Z}D(xq^m)=\sum_{m=0}^\infty\bigl(D(xq^m)-D(x^{-1}q^{m+1})\bigr)
$$
together with its companion
$$
J(q;x):=\sum_{m=0}^\infty\bigl(J(xq^m)-J(x^{-1}q^{m+1})\bigr)+\frac{\log^2|q|}3\,B_3\biggl(\frac{\log|x|}{\log|q|}\biggr),
$$
where
$$
D(x):=\ln|x|\,\arg(1-x)+\Im\Li_2(x)
=-\ln|x|\,\Im\Li_1(x)+\Im\Li_2(x)
$$
denotes the Bloch--Wigner dilogarithm and
$$
J(x):=\ln|x|\,\ln|1-x|=-\ln|x|\,\Re\Li_1(x)
$$
its companion. Namely, the expansion in the theorem can be stated as
\begin{equation}
F(z;\tau)=\frac1{2\pi i}\bigl(D(q;x)+iJ(q;x)\bigr).
\label{Fdi}
\end{equation}
This is essentially the result discussed in \cite[Section 1]{DI07}.

Viewing now $z$ as an element of the lattice $\mathbb R+\mathbb R\tau$,
so that $A$ and $\hat A$ in the representation $z=-\hat A+A\tau$ are fixed,
we find out that the $\tau$-derivative
$$
\frac1{2\pi i}\,\frac{\d}{\d\tau}\ln Q(z;\tau)
=q\frac{\d}{\d q}\ln Q(z;\tau)
$$
is the Eisenstein series
$$
\frac i{4\pi^3}\sideset{}{'}\sum_{m,n\in\mathbb Z}\frac{e^{2\pi i(m\hat A+nA)}}{(m\tau+n)^3}
$$
of weight 3, where the notation $\sum'$ indicates omitting the term $m=n=0$ from the summation.
Integrating we obtain
$$
\ln Q(z;\tau)
=\frac1{4\pi^2}\sideset{}{'}\sum_{m,n\in\mathbb Z}\frac{e^{2\pi i(m\hat A+nA)}}{m(m\tau+n)^2}
$$
implying
\begin{align*}
\ln|Q(z;\tau)|
&=\frac12\bigl(\ln Q(z;\tau)+\ln\ol{Q(z;\tau)}\,\bigr)
\\
&=\frac1{8\pi^2}\sideset{}{'}\sum_{m,n\in\mathbb Z}e^{2\pi i(m\hat A+nA)}\biggl(\frac1{m(m\tau+n)^2}-\frac1{m(m\ol\tau+n)^2}\biggr)
\\
&=\frac1{2\pi^2}\sideset{}{'}\sum_{m,n\in\mathbb Z}e^{2\pi i(m\hat A+nA)}\frac{i\,m\Im\tau\,(m\Re\tau+n)}{m(m\tau+n)^2(m\ol\tau+n)^2}
\\
&=\frac{i\,\Im\tau}{2\pi^2}\sideset{}{'}\sum_{m,n\in\mathbb Z}\frac{e^{2\pi i(m\hat A+nA)}(m\Re\tau+n)}{|m\tau+n|^4}.
\end{align*}
This is equation (7) in \cite{DI07}.
Since $\hat z=z/\tau=A-\hat A/\tau=A+\hat A\hat\tau$, it follows that
\begin{align*}
\ln|Q(\hat z;\hat\tau)|
&=\frac{i\,\Im\hat\tau}{2\pi^2}\sideset{}{'}\sum_{m,n\in\mathbb Z}\frac{e^{2\pi i(-mA+n\hat A)}(m\Re\hat\tau+n)}{|m\hat\tau+n|^4}
\displaybreak[2]\\
&=\frac{i\,\Im\tau}{2\pi^2|\tau|^2}\sideset{}{'}\sum_{m,n\in\mathbb Z}\frac{e^{2\pi i(n\hat A-mA)}(-m(\Re\tau)/|\tau|^2+n)}{|n-m/\tau|^4}
\displaybreak[2]\\
&=\frac{i\,\Im\tau}{2\pi^2}\sideset{}{'}\sum_{m,n\in\mathbb Z}\frac{e^{2\pi i(n\hat A-mA)}(n|\tau|^2-m\Re\tau)}{|n\tau-m|^4}
\displaybreak[2]\\
&=\frac{i\,\Im\tau}{2\pi^2}\sideset{}{'}\sum_{m,n\in\mathbb Z}\frac{e^{2\pi i(m\hat A+nA)}(m|\tau|^2+n\Re\tau)}{|m\tau+n|^4}
\displaybreak[2]\\
&=\frac{\Im\tau}{2\pi^2}\sideset{}{'}\sum_{m,n\in\mathbb Z}\frac{e^{2\pi i(m\hat A+nA)}((m\Re\tau+n)\tau i+(m\tau+n)\Im\tau)}{|m\tau+n|^4}
\end{align*}
implying
$$
\ln|Q(\hat z;\hat\tau)|-\tau\,\ln|Q(z;\tau)|
=\frac{(\Im\tau)^2}{2\pi^2}\sideset{}{'}\sum_{m,n\in\mathbb Z}\frac{e^{2\pi i(m\hat A+nA)}(m\tau+n)}{|m\tau+n|^4}.
$$
The latter is a (non-holomorphic) modular form of weight $1$, and combined with equation \eqref{Fdi}, is the formula of Bloch
mentioned previously.

\begin{theorem}[Bloch's formula {\cite{Bl77,DI07,Za90}}]
\label{th3}
For $z=A\tau-\hat A$, we have
\begin{align*}
F(z;\tau)
&=\frac1{2\pi i}\bigl(D(q;x)+iJ(q;x)\bigr)
\\
&=\frac{(\Im\tau)^2}{2\pi^2}\sideset{}{'}\sum_{m,n\in\mathbb Z}\frac{e^{2\pi i(m\hat A+nA)}(m\tau+n)}{|m\tau+n|^4}.
\end{align*}
\end{theorem}

\section{General weight}
\label{sec5}

A natural generalisation of the product in \eqref{Qprod} is 
\begin{equation}
Q_k(z;\tau):=q^{B_{k+2}(A)/(k+2)}\prod_{m=0}^\infty(1-xq^m)^{(m+A)^k}(1-x^{-1}q^{m+1})^{(-1)^k(m+1-A)^k},
\label{Rprod}
\end{equation}
where $k=0,1,2,\dots$ and $B_k(t)$ stands for the $k$th Bernoulli polynomial. Then $Q_0(z;\tau)$ is an arithmetic normalisation
of the short theta-function $\theta_0(z;\tau)$ (a Siegel modular unit) and $Q_1(z;\tau)$ coincides with \eqref{Qprod}.
Following the earlier recipe, define
\begin{align*}
F_+(z;\tau)&=F_{k,+}(z;\tau):=\ln Q_k(\hat z;\hat\tau)-\tau^{k-2}\ln Q_k(z;\tau),
\\
F_-(z;\tau)&=F_{k,-}(z;\tau):=\ln\ol{Q_k(\hat z;\hat\tau)}-\tau^{k-2}\ln\ol{Q_k(z;\tau)}
\end{align*}
and $F_k(z;\tau):=\frac12\bigl(F_{k,+}(z;\tau)+F_{k,-}(z;\tau)\bigr)$.
Then from Lemma~\ref{lem1} we deduce the following generalisation of Lemma~\ref{lem4}.

\begin{lemma}
\label{lem6}
We have, for $k\ge1$,
\begin{align*}
\tau^kF_+(\hat z;\hat\tau)+(-1)^kF_+(z;\tau)
&=\phantom+(-1)^k \pi iA^k(2\hat A+1)\tau^k,
\\
\tau^kF_-(\hat z;\hat\tau)+(-1)^kF_-(z;\tau)
&=-(-1)^k \pi iA^k(2\hat A+1)\tau^k.
\end{align*}
\end{lemma}

\begin{proof}
Apply Lemma \ref{lem1} and the relation
\begin{equation*}
B_{k+2}(-t)-(-1)^k B_{k+2}(t)=(-1)^k(k+2)t^{k+1}.
\qedhere
\end{equation*}
\end{proof}

We further use that the $\tau$-derivative of $\ln Q_k(z;\tau)$ is an Eisenstein series.

\begin{lemma}
\label{lem7}
For $k\ge1$,
$$
\ln Q_k(z;\tau)=\frac{(-1)^kk!}{(2\pi i)^{k+1}}\sideset{}{'}\sum_{m,n\in\mathbb Z}\frac{e^{2\pi i(m\hat A+nA)}}{m(m\tau+n)^{k+1}},
$$
where $z=-\hat A+A\tau$.
\end{lemma}

\begin{proof}
Consider $\tilde Q_k(A,\hat A;\tau):=Q_k(A\tau-\hat A;\tau)$ as a function of real variables $A,\hat A$ and complex variable $\tau$.
The $\tau$-derivative
$$
G_{k+2}(A,\hat A;\tau)
:=\frac1{2\pi i}\,\frac{\d}{\d\tau}\ln Q_k(A,\hat A;\tau)
=q\frac{\d}{\d q}\ln Q_k(A,\hat A;\tau)
$$
is seen to be the Eisenstein series
$$
E_{k+2}(A,\hat A;\tau)
:=\frac{(-1)^{k+1}(k+1)!}{(2\pi i)^{k+2}}\sideset{}{'}\sum_{m,n\in\mathbb Z}\frac{e^{2\pi i(m\hat A+nA)}}{(m\tau+n)^{k+2}}
$$
of weight $k+2$. This is true for $k=1$ (see Section~\ref{sec4}), while for $k\ge1$ we observe the functional equation
$$
\frac{\partial}{\partial\hat A}E_{k+3}(A,\hat A;\tau)
=\frac{\partial}{\partial\tau}E_{k+2}(A,\hat A;\tau).
$$
The equality $G_{k+2}(A,\hat A;\tau)=E_{k+2}(A,\hat A;\tau)$ then follows by induction on $k$ using the fact that
the constant terms of both functions at $\tau=\infty$ (or $q=0$) agree.

Integrating we obtain
$$
\ln Q_k(A,\hat A;\tau)
=\frac{(-1)^kk!}{(2\pi i)^{k+1}}\sideset{}{'}\sum_{m,n\in\mathbb Z}\frac{e^{2\pi i(m\hat A+n A)}}{m(m\tau+n)^{k+1}}.
$$
Since both sides continuously depend on $A$ and $\hat A$, the formula remains valid also for $\ln Q_k(z;\tau)$.
\end{proof}

As in our computation in Section~\ref{sec4} we obtain
\begin{align*}
\ln|Q_k(z;\tau)|
&=\frac{(-1)^kk!}{2(2\pi i)^{k+1}}\sideset{}{'}\sum_{m,n\in\mathbb Z}e^{2\pi i(m\hat A+nA)}
\biggl(\frac1{m(m\tau+n)^{k+1}}-\frac1{m(m\ol\tau+n)^{k+1}}\biggr)
\\
&=\frac{(-1)^kk!}{2(2\pi i)^{k+1}}\sideset{}{'}\sum_{m,n\in\mathbb Z}
\frac{e^{2\pi i(m\hat A+nA)}(\ol\tau-\tau)}{(m\tau+n)^{k+1}(m\ol\tau+n)^{k+1}}
\sum_{j=0}^k(m\tau+n)^j(m\ol\tau+n)^{k-j}
\\
&=-\frac{i^kk!\,\Im\tau}{(2\pi)^{k+1}}\sum_{j=0}^k
\sideset{}{'}\sum_{m,n\in\mathbb Z}
\frac{e^{2\pi i(m\hat A+nA)}}{(m\tau+n)^{k-j+1}(m\ol\tau+n)^{j+1}}
\end{align*}
and
\begin{align*}
\ln|Q_k(\hat z;\hat\tau)|
&=-\frac{i^kk!\,\Im\tau}{(2\pi)^{k+1}|\tau|^2}\sum_{j=0}^k
\sideset{}{'}\sum_{m,n\in\mathbb Z}\frac{e^{2\pi i(-mA+n\hat A)}}{(n-m/\tau)^{j+1}(n-m/\ol\tau)^{k-j+1}}
\\
&=-\frac{i^kk!\,\Im\tau}{(2\pi)^{k+1}}\sum_{j=0}^k
\sideset{}{'}\sum_{m,n\in\mathbb Z}\frac{e^{2\pi i(m\hat A+nA)}\tau^{k-j}\,{\ol\tau}^j}{(m\tau+n)^{k-j+1}(m\ol\tau+n)^{j+1}}.
\end{align*}
Thus,
\begin{align*}
F_k(z;\tau)
&=\ln|Q_k(\hat z;\hat\tau)|-\tau^k\ln|Q_k(z;\tau)|
\\
&=\frac{i^kk!\,\Im\tau}{(2\pi)^{k+1}}\sum_{j=0}^k\tau^{k-j}(\tau^j-{\ol\tau}^j)
\sideset{}{'}\sum_{m,n\in\mathbb Z}\frac{e^{2\pi i(m\hat A+nA)}}{(m\tau+n)^{j+1}(m\ol\tau+n)^{k-j+1}}
\\
&=\frac{i^kk!}{2(2\pi)^k(\tau-\ol\tau)^k}\sum_{j=1}^k\tau^{k-j}(\tau^j-{\ol\tau}^j)D_{j+1,k-j+1}(q;x)
\\
&=\frac{i\,k!}{(4\pi\Im\tau)^k}\sum_{j=1}^k\tau^{k-j}\Im(\tau^j)\,D_{j+1,k-j+1}(q;x),
\end{align*}
where
\begin{equation}
D_{a,b}(q;x):=\frac{(\tau-\ol\tau)^{a+b-1}}{2\pi i}
\sideset{}{'}\sum_{m,n\in\mathbb Z}\frac{e^{2\pi i(m\hat A+nA)}}{(m\tau+n)^a(m\ol\tau+n)^b}
\label{Dab}
\end{equation}
for positive integers $a$ and $b$.

Finally, observe that the non-holomorphic Eisenstein series \eqref{Dab} can be identified with the elliptic polylogarithms
using a formula of Zagier \cite[Proposition 2]{Za90}. This leads to the following general result.

\begin{theorem}
\label{th4}
For $k\ge1$ and $z=A\tau-\hat A$, we have
$$
\ln|Q_k(\hat z;\hat\tau)|-\tau^k\ln|Q_k(z;\tau)|
=\frac{i\,k!}{(4\pi\Im\tau)^k}\sum_{j=1}^k\tau^{k-j}\Im(\tau^j)\,D_{j+1,k-j+1}(q;x),
$$
where
$$
D_{a,b}(q;x)
=\sum_{m=0}^\infty\bigl(D_{a,b}(xq^m)+(-1)^{a+b}D_{a,b}(x^{-1}q^{m+1})\bigr)
+\frac{(4\pi\Im\tau)^{a+b-1}}{(a+b)!}\,B_{a+b}(A)
$$
and
\begin{align*}
D_{a,b}(x)
&=(-1)^{a-1}\sum_{\ell=a}^{a+b-1}2^{a+b-\ell-1}\binom{\ell-1}{a-1}\frac{(-\ln|x|)^{a+b-\ell-1}}{(a+b-\ell-1)!}\,\Li_\ell(x)
\\ &\qquad
+(-1)^{b-1}\sum_{\ell=b}^{a+b-1}2^{a+b-\ell-1}\binom{\ell-1}{b-1}\frac{(-\ln|x|)^{a+b-\ell-1}}{(a+b-\ell-1)!}\,\ol{\Li_\ell(x)}.
\end{align*}
\end{theorem}

\section{Conclusion}
\label{sec6}

This final (and very short!) part is devoted to highlighting some directions for further research.

In spite of generalisability of the story in Sections~\ref{sec2}--\ref{sec4} to the function
$$
F_k(z;\tau)
=\ln|Q_k(\hat z;\hat\tau)|-\tau^k\ln|Q_k(z;\tau)|,
$$
where $k\ge1$ and the product $Q_k(z;\tau)$ is defined in~\eqref{Rprod}, the case $k=1$ remains
the only one, which is invariant under translation $\tau\mapsto\tau+1$.
At the same time, Lemma~\ref{lem6} implies the transformation
$$
\tau^kF_k(\hat z,\hat\tau)=(-1)^{k-1}F_k(z,\tau)
\quad\text{for}\; k=1,2,\dotsc.
$$
This consideration does not exclude, however, a possibility for modified products \eqref{Rprod} and
related functions $F_k$ to exist such that the latter ones have true modular behaviour for each $k\ge1$.
It sounds to us a nice problem to determine such modular objects.

Several arithmetic problems related to the case $k=1$ (originating from the elliptic gamma function) are still open.
Our personal favourites include connection of \eqref{Qprod} with the Mahler measure and mirror symmetry;
see, for example, Observation in~\cite{St06}.

\section*{Acknowledgements}
We thank the anonymous referees for their careful reading of the manuscript and reporting valuable feedback.



\end{document}